\documentclass{article}
\usepackage[nottoc]{tocbibind}
\usepackage[titletoc,title]{appendix}
\usepackage[utf8]{inputenc}

\usepackage[pdftex]{graphicx}
\usepackage[pdftex,dvipsnames,usenames]{color}

\usepackage{enumitem}
\usepackage{mathtools}

\usepackage{cancel}

\usepackage[english]{babel}
\usepackage[hyperindex,breaklinks]{hyperref}

\hypersetup{
  colorlinks   = true, 
  urlcolor     = blue, 
  linkcolor    = blue, 
  citecolor   = red 
}

\usepackage{amsmath,amssymb,amsfonts,mathrsfs, amsthm}

\usepackage{graphicx}
\usepackage[margin=0.75in]{geometry}

\usepackage{soul}

\usepackage{pdfpages}

\usepackage{lipsum}
\usepackage{csquotes}
\usepackage{thmtools}

\usepackage[
  backend=biber,
  style=alphabetic,
  maxnames=5,       
  minnames=5,       
  maxbibnames=5,    
  minbibnames=5,     
  maxalphanames=5,  
  minalphanames=5
]{biblatex}
\usepackage{amsfonts,amsmath,amstext,amssymb,amsthm}
\usepackage{enumitem}
\usepackage[english]{babel}
\usepackage{fancybox}
\usepackage{xspace}
\usepackage{amscd}
\usepackage[all]{xy}
\usepackage[svgnames]{xcolor}
\usepackage{calc}

\usepackage{esint}

\usepackage{enumitem}

\usepackage{cases}
\usepackage{eqnarray}

\usepackage[compact]{titlesec}

\usepackage{accents}

\def\XXint#1#2#3{{\setbox0=\hbox{$#1{#2#3}{\int}$ }
\vcenter{\hbox{$#2#3$ }}\kern-.56\wd0}}

\newcommand{\eps}{\varepsilon}

\newcommand{\N}{\mathbb{N}}
\newcommand{\R}{\mathbb{R}}

\newcommand{\ep}{\varepsilon}

\usepackage{cleveref}

\newtheorem{theorem}{Theorem}[section]

\newtheorem{lemma}[theorem]{Lemma}

\theoremstyle{definition}

\newtheorem{remark}[theorem]{Remark}

\usepackage{url}

\makeatletter
\renewenvironment{abstract}{
  \par\smallskip
  \noindent\textsc{Abstract.}\ %
}{
  \par\medskip
}
\makeatother

\makeatletter
\newcommand{\addressa}[1]{\gdef\@addressa{#1}}
\newcommand{\emaila}[1]{\gdef\@emaila{\url{#1}}}

\newcommand{\@endstuff}{\par\vspace{\baselineskip}\noindent
\begin{tabular}{@{}l}\scshape\@addressa\\\textit{E-mail address:} \@emaila\end{tabular} 

}

\AtEndDocument{\@endstuff}
\makeatother

\begin{document}

\setlength{\abovedisplayskip}{6pt}          
\setlength{\belowdisplayskip}{6pt}          
\setlength{\abovedisplayshortskip}{6pt}     
\setlength{\belowdisplayshortskip}{6pt}     

\setlength{\parskip}{0pt plus 1pt}
\setlength{\parindent}{15pt}
\setlength{\abovedisplayskip}{5pt plus 1pt minus 2pt}
\setlength{\belowdisplayskip}{5pt plus 1pt minus 2pt}

\setlist[itemize]{itemsep=2pt, topsep=3pt, partopsep=1pt, parsep=1pt}
\setlist[enumerate]{noitemsep, topsep=2pt, partopsep=0pt, parsep=0pt}

\titlespacing{\section}{0pt}{10pt plus 2pt minus 2pt}{6pt plus 1pt minus 1pt}
\titlespacing{\subsection}{0pt}{8pt plus 2pt minus 2pt}{5pt plus 1pt minus 1pt}

\title{\huge 
Equivalence of intrinsic and extrinsic area bounds for minimal surfaces}

\author{Enric Florit-Simon}
\date{}

\begingroup
\renewcommand\thefootnote{}
\footnotetext{E. F. was supported by the European Research Council under the Grant Agreement No 948029.}
\endgroup

\maketitle

\begin{abstract}
We show that intrinsic and extrinsic area density bounds are equivalent, with matching asymptotic values, for complete, connected, smooth minimal immersions $\Sigma^d\overset{i}{\hookrightarrow}\R^N$ of any dimension and codimension. Combining our results with a recent breakthrough by Bellettini \cite{Bel25}, we extend the Schoen--Simon--Yau curvature estimates \cite{SSY75} for smoothly immersed, two-sided, stable minimal hypersurfaces $\Sigma^n\overset{i}{\hookrightarrow}\R^{n+1}$ with bounded intrinsic area density to the missing case $n=6$, which had remained open since.
\end{abstract}

\addressa{Enric Florit-Simon \\ Department of Mathematics, ETH Z\"{u}rich \\ Rämistrasse 101, 8092 Zürich, Switzerland}
\emaila{enric.florit@math.ethz.ch}
\newcommand{\enric}[1]{{\color{green}{#1}}}

\section{Introduction}\label{sec:intro}

Let $\Sigma^d\overset{i}{\hookrightarrow}\R^N$ be an immersed minimal submanifold of dimension $d$.
Given $p\in\Sigma$ and $R>0$, define its \textit{intrinsic} and \textit{extrinsic area densities} via
$${\bf M}_R^{{\rm int}}(\Sigma,p):=\frac{1}{R^d}|B_R^\Sigma(p)|\qquad\mbox{and}\qquad {\bf M}_R^{{\rm ext}}(\Sigma,p):=\frac{1}{R^d}{\rm Area}(\Sigma \cap B_R^{\R^N}(p)):=\frac{1}{R^d}|i^{-1}(B_R^{\R^N}(p))|.$$
Here $B_R^\Sigma(p)=\{x\in\Sigma:d_\Sigma(x,p)<R\}$ denotes the intrinsic ball of radius $R$, and all lengths and areas on $\Sigma$ are computed with respect to the pullback metric $i^*(g_{\R^n})$. In particular, extrinsic area is ``counted with multiplicity'', as $i$ need not be injective, and it corresponds to the natural varifold measure associated to the immersion.\\

Our main result is the equivalence between intrinsic and extrinsic area density bounds for entire minimal immersions of any dimension and codimension:
\begin{theorem}\label{thm:extrbound}
    Let $\Sigma^d\overset{i}{\hookrightarrow} \R^N$ be a complete, connected, smoothly immersed minimal submanifold, and let $p\in\Sigma$. Then ${\bf M}_R^{\rm int}(\Sigma,p)$ and ${\bf M}_R^{\rm ext}(\Sigma,p)$, which are both monotone nondecreasing, satisfy
    $$\lim_{R\to\infty} {\bf M}_R^{\rm int}(\Sigma,p)=\lim_{R\to\infty} {\bf M}_R^{\rm ext}(\Sigma,p)\in[\omega_d,\infty],$$
    where $\omega_d:=|B_1^{\R^d}|$.

    In particular, $\Sigma$ has bounded intrinsic area density if and only if it has bounded extrinsic area density. In such a case, $\Sigma$ is additionally a proper immersion.
\end{theorem}
\begin{remark}\label{rmk:trivbound}
    The inequality ${\bf M}_R^{\rm int}(\Sigma,p)\leq {\bf M}_R^{\rm ext}(\Sigma,p)$ is trivial and completely local, by the direct comparison $|p-q|\leq d_{\Sigma}(p,q)$ for any $p,q\in\Sigma\overset{i}{\hookrightarrow}\R^N$. Our proof consists in finding an asymptotic chord-arc estimate in the opposite direction, exploiting a monotonicity formula for the intrinsic area density.
\end{remark}
A main motivation behind this result is the lack of a general compactness theory for minimal immersions (or even embeddings) with intrinsic area bounds. This is in stark contrast with the case of extrinsically bounded area, where stationary integral varifolds---and their associated compactness theorem \cite{All72}---provide the main tool.

In fact, even for stable, codimension one minimal immersions $\Sigma^n\overset{i}{\hookrightarrow} \R^{n+1}$, this issue has not as of yet been resolved. This is exemplified by a well-known gap between the best known (optimal) dimensional requirement $n\leq 6$ for the flatness of entire, stable, codimension one minimal embeddings with extrinsic area bounds \cite{SS81} and that for immersions (or even embeddings) with only intrinsic area bounds \cite{SSY75}, in which the case $n=6$ has been missing since. In combination with the recent breakthrough \cite{Bel25}, which removes the embeddedness requirement for $n=6$, \Cref{thm:extrbound} resolves the dimensional gap:
\begin{theorem}\label{thm:SSYR7}
    Let $\Sigma^6\overset{i}{\hookrightarrow} \R^7$ be a complete, connected, smoothly immersed, two-sided stable minimal hypersurface. Assume that $\Sigma$ has bounded intrinsic area density, i.e. there are some $p\in\Sigma$ and $\Lambda\in(0,\infty)$ such that $|B_R^\Sigma(p)|\leq \Lambda R^6$ for all $R>0$. Then $\Sigma$ is a flat hyperplane.
\end{theorem}
For $n\geq 7$, \cite{BDGG69} provided a nonflat counterexample (in fact, an area-minimising boundary of a set of locally finite perimeter), asymptotic to Simons' cone \cite{Simons68}. In this sense, \Cref{thm:SSYR7} is sharp.

Local curvature estimates for stable minimal immersions will follow, completing the picture in \cite{SSY75}:
\begin{theorem}\label{thm:curvests}
    Let $\Sigma^6\overset{i}{\hookrightarrow} \R^7$ be a complete, connected, smoothly immersed, two-sided stable minimal hypersurface with boundary $\partial \Sigma$. Assume that $d_\Sigma(p,\partial \Sigma)>R$ and $|B_R^\Sigma(p)|\leq \Lambda R^6$ for some $p\in\Sigma$ and $\Lambda\,,R>0$. Then, 
    $$|\mathrm{II}_\Sigma|\leq \frac{C}{R}\qquad \mbox{in}\quad B_{R/2}^\Sigma(p)$$
    for some $C=C(\Lambda)>0$.
\end{theorem}

Let us recall the historical progress towards this result. Let $\Sigma^{n}\overset{i}{\hookrightarrow} \R^{n+1}$ be as in \Cref{thm:curvests}.
\begin{itemize}
    \item For $n\leq 5$, \cite{SSY75} proved the exact analogue result.
    \item For $n=6$, under the additional assumption that $\Sigma$ be an area-minimising boundary, the local curvature estimates were obtained by \cite{Simon76}.
    \item For $n = 6$, under the additional assumptions of embeddedness and a uniform bound on the extrinsic area density, they were obtained in \cite{SS81}.
    \item Much more recently, for $n = 6$ \cite{Bel25} removed the embeddedness condition, obtaining the local curvature estimates under the only additional assumption of a uniform bound on the extrinsic area density.
\end{itemize}
Interestingly, unlike \Cref{thm:SSYR7}, the local curvature estimates in \Cref{thm:curvests} do not follow directly from the ones in \cite{Bel25}, as the comparability between intrinsic and extrinsic area densities in \Cref{thm:extrbound} is exclusively global. We will reduce \Cref{thm:curvests} to \Cref{thm:SSYR7} instead, by combining a well-known scaling argument and the monotonicity of intrinsic area density.
\begin{remark}
    The well-known ``stable Bernstein conjecture'' asks whether the same results hold even without any assumption on the area of $\Sigma$. This was shown to be true in \cite{FS80, DP79, Pog81} for $n=2$, and for $n=3,4,5$ in very recent breakthroughs \cite{CL24, CLMS25, Maz24} (see also \cite{CL23, CMR24, CCMMA26, Str26}). A main step in several of the proofs, particularly in higher dimensions, is to show that any complete, connected, smoothly immersed, two-sided stable minimal hypersurface $\Sigma^n\overset{i}{\hookrightarrow}\R^{n+1}$ satisfies intrinsic area bounds indeed, thus reducing their classification as flat hyperplanes to the main result in \cite{SSY75}. \Cref{thm:SSYR7} suggests this as a viable strategy also for $n=6$, removing a well-known obstruction in the missing case.
\end{remark}
\begin{remark}
    In combination with \cite{Bel25}, this also provides a new proof of the results in \cite{SSY75} for $\Sigma^n\overset{i}{\hookrightarrow} \R^{n+1}$ with $n\leq 5$ (either by running the same arguments, or simply considering $\Sigma^n\times\R^{6-n}\overset{i}{\hookrightarrow} \R^{7}$ in \Cref{thm:SSYR7,thm:curvests}).
\end{remark}
\section{Proofs}\label{sec:proofthm}
In what follows we always assume completeness, connectedness and smoothness of any $\Sigma^d\overset{i}{\hookrightarrow}\R^N$. We work in arbitrary dimension and codimension, and we do not require stability unless otherwise indicated. The notation ``$0\in\Sigma$'' will mean that we fix some point $o\in\Sigma$, denoted by $0$ by slight abuse of notation, such that  $i(o)=0$. Having fixed such a choice, we denote $r_\Sigma(x):={\rm dist}_\Sigma(x,o)$, ${\bf M}_R^{\rm int}(\Sigma):={\bf M}_R^{\rm int}(\Sigma,o)$, and $B_R^\Sigma:=B_R^\Sigma(o)$.

The main ingredient in the proof is a monotonicity formula for the intrinsic area density (motivated by a computation in \cite[Section 7]{Yau75} to obtain lower volume estimates for minimal submanifolds in nonpositive curvature, which we learnt from \cite[Lemma 8.9]{Cho21notes}). Although likely folklore, and recently noted in the literature (see \cite[Proof of Proposition 2.4]{MP24}), it seems to have received significantly less attention than its extrinsic counterpart; in particular, we have not found a reference for its simple (albeit powerful) error term.
\begin{lemma}[Monotonicity formula]\label{lem:monformula}
    Let $\Sigma^d\overset{i}{\hookrightarrow}\R^N$ be an immersed submanifold with boundary $\partial\Sigma$ and $0\in\Sigma$. Let ${\rm H}_\Sigma$ denote its mean curvature vector. Then, given $0<R_1<R_2<{\rm dist}_\Sigma(0,\partial \Sigma)$, we have
    \begin{equation}\label{eq:monfor}
        {\bf M}_{R_2}^{\rm int}(\Sigma)-{\bf M}_{R_1}^{\rm int}(\Sigma)=
\int_{B_{R_2}^\Sigma\setminus B_{R_1}^\Sigma} \frac{(r_\Sigma(x)-|x|)+\frac{|x|}{2}(\frac{x}{|x|}-\nu_{B_{r_\Sigma(x)}^\Sigma}(x))^2}{r_\Sigma^{d+1}(x)}\,dx+\int_{R_1}^{R_2}\int_{B_{t}^\Sigma} \frac{x\cdot{\rm H}_\Sigma(x)}{t^{d+1}}\,dx\,dt\,.
\end{equation}
\end{lemma}
In particular, if $\Sigma$ is minimal then ${\bf M}_R^{\rm int}(\Sigma)$ is monotone nondecreasing in $R\in(0,{\rm dist}_\Sigma(0,\partial \Sigma))$.
\begin{proof}
    A simple computation shows that $\Delta_\Sigma |x|^2=2d+ 2x\cdot{\rm H}_\Sigma$. Integrating and using the divergence theorem, letting $\nu_{B_r^\Sigma}$ denote the exterior intrinsic normal vector to the ball we find
    \begin{align*}
    d\,|B_R^\Sigma|+\int_{B_R^\Sigma}x\cdot{\rm H}_\Sigma&=\frac{1}{2}\int_{B_R^\Sigma} \Delta_\Sigma |x|^2=\int_{\partial B_R^\Sigma} x\cdot \nu_{B_R^\Sigma}(x).
        \end{align*}
        Writing
        $$x\cdot \nu_{B_R^\Sigma}(x)=R-\big[(R-|x|)+(|x|-x\cdot\nu_{B_R^\Sigma}(x))\big]=R-\big[(R-|x|)+\frac{|x|}{2}(\frac{x}{|x|}-\nu_{B_R^\Sigma}(x))^2\big]$$
        we find
        \begin{align*}
        d\,|B_R^\Sigma|+\int_{B_R^\Sigma}x\cdot{\rm H}_\Sigma&=R|\partial B_R^\Sigma|-\int_{\partial B_R^\Sigma} \big[(R-|x|)+\frac{|x|}{2}(\frac{x}{|x|}-\nu_{B_R^\Sigma}(x))^2\big]dx\,.
    \end{align*}
    Since
    $$R^{d+1}\frac{d}{dR}(R^{-d}|B_R^\Sigma|)=R|\partial B_R^\Sigma|-d|B_R^\Sigma|,$$
    dividing both sides by $R^{d+1}$, using that $R=r_\Sigma(x)$ for every $x\in\partial B_R^\Sigma$, and
    integrating we get \eqref{eq:monfor}.
    
    If $\Sigma$ is minimal (and thus ${\rm H_\Sigma}\equiv 0$) the right-hand side in \eqref{eq:monfor} is nonnegative, since $r_\Sigma(x)\geq |x|$ by \Cref{rmk:trivbound}, hence ${\bf M}_R^{\rm int}(\Sigma)$ is then monotone nondecreasing in $R$.
\end{proof}
\begin{remark}\label{rmk:lowdens}
     In particular, if $\Sigma$ is minimal, recentering \Cref{lem:monformula} and sending $R_1\to 0$ we get the lower density bound
    $${\bf M}_R^{\rm int}(\Sigma,p)\geq |B_1^{\R^d}|=\omega_d\qquad \forall\, p\in \Sigma\quad\mbox{and}\quad R\in(0,{\rm dist}_\Sigma(0,\partial\Sigma)).$$
\end{remark}
\begin{lemma}[Density pinching implies chord-arc estimate]\label{lem:chordarc}
    Let $\Sigma^d\overset{i}{\hookrightarrow}\R^N$ be an immersed minimal submanifold with $0\in\Sigma$. Given $\eps\in(0,1)$, there exists $\delta=\delta(\eps,d)>0$ such that: If ${\bf M}_{8R}^{\rm int}(\Sigma)-{\bf M}_R^{\rm int}(\Sigma)\leq \delta$, then 
    $$|x|\leq r_\Sigma(x)\leq (1+\ep)|x|\qquad \mbox{in} \quad B_{4R}^\Sigma\setminus B_{2R}^\Sigma.$$
\end{lemma}
\begin{proof}
By scaling invariance we can assume that $R=1$.

    The inequality $|x|\leq r_\Sigma(x)$ is trivial (recall \Cref{rmk:trivbound}). Assume, for contradiction, that there were $x\in B_{4}^\Sigma\setminus B_{2}^\Sigma$ with $r_\Sigma(x)\geq (1+\eps)|x|$, thus with $r_\Sigma(x)-|x|\geq \frac{\eps}{1+\eps}r_\Sigma(x)\geq \frac{2\eps}{1+\eps}$. By the triangle inequality (both intrinsic and Euclidean), for every $y\in B_{\frac{\eps}{2(1+\eps)}}^\Sigma(x)\subset B_8^\Sigma\setminus B_1^\Sigma$ we still have
    \begin{equation}\label{eq:81tguqwp}
        r_\Sigma(y)-|y|\geq \big(r_\Sigma(x)-\frac{\eps}{2(1+\eps)}\big)-\big(|x|+\frac{\eps}{2(1+\eps)}\big)\geq \frac{\eps}{(1+\eps)}.
    \end{equation}
    On the other hand, \eqref{eq:monfor} shows that
    $$
    {\bf M}_8^{\rm int}(\Sigma)-{\bf M}_1^{\rm int}(\Sigma)\geq \int_{B_8^\Sigma\setminus B_1^\Sigma} \frac{r_\Sigma(y)-|y|}{r_\Sigma^{d+1}(y)}\,dy\geq \frac{1}{8^{d+1}}\int_{B_{\frac{\eps}{2(1+\eps)}}^\Sigma(x)} (r_\Sigma(y)-|y|)\,dy\,,
    $$
    thus combined with \eqref{eq:81tguqwp} we find
    $$
    {\bf M}_8^{\rm int}(\Sigma)-{\bf M}_1^{\rm int}(\Sigma) \geq \frac{\eps}{8^{d+1}(1+\eps)}|B_{\frac{\eps}{2(1+\eps)}}^\Sigma(x)|.
    $$
    By \Cref{rmk:lowdens} we can bound $|B_{\frac{\eps}{2(1+\eps)}}^\Sigma(x)|\geq \frac{\eps^d}{2^d(1+\eps)^d}|B_1^{\R^d}|$, thus
    $$
    {\bf M}_8^{\rm int}(\Sigma)-{\bf M}_1^{\rm int}(\Sigma) \geq \frac{\eps^{d+1}}{2^{4d+3}(1+\eps)^d}|B_1^{\R^d}|,
    $$
    reaching a contradiction for $\delta$ smaller than this value.
\end{proof}
We can now give:
\begin{proof}[Proof of \Cref{thm:extrbound}]
The monotonicity of ${\bf M}_R^{\rm ext}(\Sigma)$ is well known \cite{All72}, and \Cref{lem:monformula} (see also \cite[Proof of Proposition 2.4]{MP24}) proves the monotonicity of ${\bf M}_R^{\rm int}(\Sigma)$. In particular, the existence of both limits as $R\to\infty$ (allowing for the value $+\infty$) follows. From ${\bf M}_R^{\rm int}(\Sigma)\leq {\bf M}_R^{\rm ext}(\Sigma)$ (recall \Cref{rmk:trivbound}) and ${\bf M}_R^{\rm int}(\Sigma)\geq \omega_d$ (recall \Cref{rmk:lowdens}), we also find 
$$\omega_d\leq \lim_{R\to\infty} {\bf M}_R^{\rm int}(\Sigma)\leq \lim_{R\to\infty} {\bf M}_R^{\rm ext}(\Sigma).$$
If $\lim_{R\to\infty}{\bf M}_R^{\rm int}(\Sigma)=\infty$, there is nothing left to prove; assume then that $\lim_{R\to \infty} {\bf M}_R^{\rm int}(\Sigma)<\infty$ instead. Let $\eps\in(0,1)$, and let $\delta=\delta(d,\eps)>0$ be given by \Cref{lem:chordarc}. Pick $R_0=R_0(\eps,\Sigma)>0$ such that $\lim_{R\to\infty} {\bf M}_{R}^{\rm int}(\Sigma)-{\bf M}_{R_0}^{\rm int}(\Sigma)\leq \delta$; by monotonicity, for every $R>R_0$ we have
${\bf M}_{8R}^{\rm int}(\Sigma)-{\bf M}_{R}^{\rm int}(\Sigma)\leq \delta$ as well. \Cref{lem:chordarc}---applied for
every $R>R_0$---shows then that $|x|\leq r_\Sigma(x)\leq (1+\eps)|x|$ for every $x\in \Sigma\setminus B_{2R_0}^\Sigma$, thus
$B_{R}^\Sigma\subset i^{-1}(B_{R}^{\R^N}) \subset B_{(1+\eps)R}^\Sigma$ for every $R\geq 2R_0$. Properness of the immersion follows (say, fixing the value $\ep=1/2$), as the preimage under $i:\Sigma\hookrightarrow\R^n$ of any compact set of $\R^n$ is then contained in some intrinsic ball, hence it is also compact. Moreover, we deduce that 
$${\bf M}_{R}^{\rm ext}(\Sigma)=R^{-d}|i^{-1}(B_{R}^{\R^N})|\leq R^{-d}|B_{(1+\eps)R}^\Sigma|= (1+\eps)^d{\bf M}_{(1+\eps)R}^{\rm int}(\Sigma),$$
thus
$$\lim_{R\to\infty}{\bf M}_{R}^{\rm ext}(\Sigma)\leq (1+\eps)^d\lim_{R\to\infty}{\bf M}_{R}^{\rm int}(\Sigma).$$
Since $\eps\in(0,1)$ was arbitrary, sending $\eps\to 0$ concludes the proof.
\end{proof}

\begin{proof}[Proof of \Cref{thm:SSYR7}]
By \Cref{thm:extrbound}, $\Sigma$ also has bounded extrinsic area density, with ${\rm Area}(\Sigma\cap B_R^{\R^7}(p))\leq \Lambda R^6$ for every $R>0$ as well, and it is properly immersed.
\cite[Theorem 1]{Bel25} then shows that $\Sigma$ is a flat hyperplane.
\end{proof}
\begin{proof}[Proof of \Cref{thm:curvests}]
By scaling invariance, we can assume that $R=1$.

Assume, for contradiction, that the result were false. There would then be a sequence of immersions $\Sigma_k^6\overset{i_k}{\hookrightarrow} \R^7$ as in the statement, with---after a translation---$0\in\Sigma_k$, and points $x_k\in B_{1/2}(\Sigma_k)$ such that $|\mathrm{II}_{\Sigma_k}(x_k)|\to\infty$.
In particular,
$$\max_{y\in B_{3/4}^{\Sigma_k}}{\rm dist}_{\Sigma_k}(y,\partial B_{3/4}^{\Sigma_k})|\mathrm{II}_{\Sigma_k}(y)|\to\infty.$$
Let $y_k\in B_{3/4}^{\Sigma_k}$ with
$${\rm dist}_{\Sigma_k}(y_k,\partial B_{3/4}^{\Sigma_k})|\mathrm{II}_{\Sigma_k}(y_k)|=\max_{y\in B_{3/4}^{\Sigma_k}}{\rm dist}_{\Sigma_k}(y,\partial B_{3/4}^{\Sigma_k})|\mathrm{II}_{\Sigma_k}(y)|\to\infty,$$
and put $R_k:={\rm dist}_{\Sigma_k}(y_k,\partial B_{3/4}^{\Sigma_k})|\mathrm{II}_{\Sigma_k}(y_k)|\to\infty$ and $\lambda_k:=|\mathrm{II}_{\Sigma_k}(y_k)|$. Consider the ``translated and rescaled'' immersion $\widetilde \Sigma_k\overset{\widetilde i_k}{\hookrightarrow} \R^7$, given by the same underlying manifold $\Sigma_k$ with new metric $g_{\widetilde \Sigma_k}=\lambda_k^2g_{\Sigma_k}$, and new immersion map $\widetilde i_k(x):=i_k(x)-i_k(y_k)$ (so that $\widetilde i_k(y_k)=0$). It follows that
$$|\mathrm{II}_{\widetilde \Sigma_k}(y_k)|=1,\quad {\rm dist}_{\widetilde \Sigma_k}(y_k,\partial \widetilde \Sigma_k)\geq R_k\to\infty,\quad \mbox{and}\quad |\mathrm{II}_{\widetilde \Sigma_k}(y)|\leq 2 \ \mbox{ in } \ B_{R_k/2}^{\widetilde \Sigma_k}(y_k).$$
Moreover, by the assumption $|B_1^{\Sigma_k}|\leq \Lambda$ and the intrinsic monotonicity formula it follows that
$$r^{-6}|B_r^{\Sigma_k}(y_k)|\leq (1/4)^{-6}|B_{1/4}^{\Sigma_k}(y_k)|\leq (1/4)^{-6}\Lambda\quad \mbox{for every} \quad r\in(0,{\rm dist}_{\Sigma_k}(y_k,\partial B_{3/4}^{\Sigma_k})),$$
which upon rescaling gives $r^{-6}|B_r^{\widetilde \Sigma_k}(y_k)|\leq (1/4)^{-6}\Lambda$ for every $r\in (0,R_k)$. In particular, given $M>0$, for $k$ large enough so that $R_k>M$ we have $|B_M^{\widetilde \Sigma_k}(y_k)|\leq (1/4)^{-6}\Lambda M^6$, a uniform bound.\\
Finally, the uniform second fundamental form estimates and standard bootstrap estimates for minimal graphs give some universal $r_0>0$ such that: for any $p\in B_{R_k/4}^{\widetilde \Sigma_k}(y_k)$, we can write $B_{r_0}^{\widetilde \Sigma_k}(p)$ as an extrinsic graph with uniform $C^{l}$ estimates for every $l\in\N$.

Combining all of the above, the standard Arzel\`a--Ascoli-type theorem for immersions with uniform $C^l$ bounds, see e.g. \cite{Del00}, \cite{Breu15}, or \cite[Proposition 8.2]{Cho21notes}, shows the existence of some limiting $\widetilde \Sigma_\infty^6\overset{\widetilde i_\infty}{\hookrightarrow} \R^7$, which is a complete, connected, smoothly immersed, two-sided stable minimal hypersurface, with $|B_R^{\widetilde \Sigma_\infty}(p)|\leq (1/4)^{-6}\Lambda R^6$ for some $p\in \widetilde \Sigma_\infty$ (with $i_\infty(p)=0$ and $|\mathrm{II}_{\widetilde \Sigma_\infty}(p)|=1$) and every $R>0$. On the other hand, \Cref{thm:SSYR7} shows that $\widetilde \Sigma_\infty$ is a hyperplane, thus $|\mathrm{II}_{\widetilde \Sigma_\infty}(p)|=0$, a contradiction.
\end{proof}


\addcontentsline{toc}{section}{\large References}
\printbibliography

\end{document}